\documentclass[a4paper,11pt]{article}
\usepackage{amsfonts}
\usepackage{amsthm}
\usepackage[all]{xy}
\usepackage{amssymb}
\usepackage[english]{babel}
\usepackage{graphicx}
\usepackage{amsmath}
\usepackage{ccaption}
\usepackage{multirow}
\usepackage{latexsym}
\newcommand{\spa}{\operatorname{span}}
\newcommand{\ke}{\operatorname{ker}}
\newcommand{\I}{\operatorname{Image}}
\newcommand{\Pic}{\operatorname{Pic}}
\newcommand{\Ann}{\operatorname{Ann}}

\begin{document}

\newtheorem{teo}{Theorem}[section]
\newtheorem{lem}[teo]{Lemma}
\newtheorem{pro}[teo]{Proposition}
\newtheorem{cor}[teo]{Corollary}

\theoremstyle{definition}
\newtheorem{oss}[teo]{Remark}

\addcontentsline{}{}{}
\setcounter{secnumdepth}{-1}

\setcounter{secnumdepth}{4}

\title{A note on Griffiths infinitesimal invariant for curves} \author{Emanuele Raviolo}
\date{}
\maketitle

\begin{abstract} 
Given a generic curve of genus $g\geqslant4$ and a smooth point $L\in W_{g-1}^{1}(C)$, 
whose linear system is base-point free, we consider  the Abel-Jacobi normal function 
associated to $L^{\otimes 2}\otimes \omega_{C}^{-1}$, 
when $(C,L)$ varies in moduli. We prove that its infinitesimal invariant reconstructs the couple $(C,L)$. 
When $g=4$, we obtain the generic Torelli theorem proved by Griffiths. 
\end{abstract}

\section*{Introduction}\label{intro}
The infinitesimal invariant of normal function was introduced by Griffiths in \cite{grif}. 
In that paper he gave a beautiful application of it, which we briefly recall.
\\
If $C$ is a generic curve of genus $4$, then its canonical image is the intersection of the smooth quadric $Q$ with a cubic in $\mathbb{P}^{3}$. 
The rulings of $Q$ cut on the canonical curve two complete $g_{3}^{1}$'s, $|L|$ and $|L^{\prime}|$, which are adjoint, i.e. 
$L^{\prime}=\omega_{C}\otimes L^{-1}$. 
Since the difference $L\otimes (L^{'})^{-1}$ has degree zero, when $C$ varies in moduli we get a normal function $\nu$ given by the Abel-Jacobi map. 
Griffiths proved that the infinitesimal invariant of $\nu$ gives the equation of the canonical curve inside $Q$.
\\
The infinitesimal invariant was then studied by M. Green \cite{g} and C. Voisin \cite{voi}.
Green refined Griffiths' original idea defining a series of infinitesimal invariants and obtained a result on the Abel-Jacobi map
of odd-dimensional projective hypersurfaces of large degree.
Voisin instead gave a geometric interpretation of the infinitesimal invariant in terms of algebraic cycles.  
\\
A very interesting application of the infinitesimal invariant to the study of algebraic cycles is due to A. Collino and G.P. Pirola \cite{cp}. 
They computed the infinitesimal invariant of the Ceresa cycle $C-C^{-}$ of a curve in its Jacobian. 
As consequences, they reproved that $C-C^{-}$ is not algebraically trivial (originally proved by G. Ceresa \cite{ceresa}) 
and obtained a generic Torelli theorem for curves of genus $3$. 
In particular they showed that the infinitesimal invariant gives the equation of the canonical curve in the projective plane.
\\
The works mentioned above show that the infinitesimal invariant is a powerfull tool in the study of algebraic cycles and in Torelli-type
problems. It encodes in fact some trascendental and algebraic information at the same time.
\\
In this paper we give the following generalisation of Griffiths result for curves.
Let $C$ be a generic curve of genus $g\geqslant 4$. 
Let us consider a non singular point $L\in W_{g-1}^{1}(C)$ whose linear system is base-point free.  
Then also the adjoint bundle $\omega_{C}\otimes L^{-1}\in W_{g-1}^{1}(C)$ is a smooth point. 
The image of the Petri map $H^{0}(C,L)\otimes H^{0}(C,\omega_{C}\otimes L^{-1})\rightarrow H^{0}(C,\omega_{C})$ is a four dimensional vector subspace 
$V$ of $H^{0}(C,\omega_{C})$ whose linear system is base-point-free.
If $\phi:C\rightarrow\mathbb{P}V^{*}\cong\mathbb{P}^{3}$ is the induced holomorphic map, then $\phi(C)$ is contained in a quadric of rank four and is 
birational to $C$.
When we deform the couple $(C,L)$, the 0-degree line bundle $L^{\otimes 2}\otimes \omega_{C}^{-1}$ gives
 a normal function $\nu$ (see Section~\ref{ex}). 
\\
We show that the infinitesimal invariant of $\nu$ reconstructs the curve $\phi(C)$ 
and, as consequence, the couple $(C,L)$ (Corollary \ref{lb}).
\\
The paper consists of two sections. The first contains a quick review of the definition of the infinitesimal invariant 
and of the results by Griffiths and Voisin that we need in our computations.
In the second we compute the infinitesimal invariant of the normal function.

\section{The infinitesimal invariant for curves}\label{1}
In this section we recall the basic facts on the infinitesimal invariant of normal functions for curves. A great reference is \cite[chap. 7]{voi2}.
\\
Consider a smooth curve $C$ and its Kuranishi family $\pi:\mathcal{C}\rightarrow B$. We define $C_{t}=\pi^{-1}(t)$ and $C=C_{0}$ 
for a reference point $0\in B$. 
 We have the associated jacobian fibration 
 \[ j(\pi):\frac{\mathcal{H}}{\mathcal{F}+R^{1}\pi_{*}\mathbb{Z}}\rightarrow B, \] 
 where $\mathcal{H}$ and $\mathcal{F}$ are the holomorphic vector bundles over $B$ with fibers $\mathcal{H}_{t}=H^{1}(C_{t},\mathbb{C})$ 
 and $\mathcal{F}_{t}=H^{0}(C_{t},\omega_{C_{t}})$ over $t\in B$.
\\
Let $\mathcal{D}\subset\mathcal{C}$ be a curve such that for every $t\in B$ the intersection divisor $D_{t}=\mathcal{D}\cdot C_{t}$ 
has degree zero on $C_{t}$.  
Then we can define a normal function 
\[\nu:B\rightarrow \frac{\mathcal{H}}{\mathcal{F}+R^{1}\pi_{*}\mathbb{Z}}\] 
setting $\nu(t)=AJ_{C_{t}}(D_{t})$, where $AJ$ is the Abel-Jacobi map.
\\
Following Mark Green \cite{g} we define the infinitesimal invariant of $\nu$ in the following way.
The Gauss Manin connection $\nabla$ of $\mathcal{H}$ induces a morphism of vector bundles \[
\nabla: \mathcal{F}\rightarrow \mathcal{H}^{0,1}\otimes\Omega_{B}^{1},\]
 where $\mathcal{H}^{0,1}=\frac{\displaystyle{\mathcal{H}}}{\displaystyle{\mathcal{F}}}$ 
 (note that Serre duality induces an isomorphism \mbox{$\mathcal{H}^{0,1}\cong\mathcal{F}^{*}$}).
\\
If $\widetilde{\nu}:U\subset B\rightarrow \mathcal{H}$ is a local lifting of $\nu$, the class 
\[ [\nabla\widetilde{\nu}]\in\frac{\mathcal{H}^{0,1}\otimes\Omega^{1}_{U}}{\nabla\mathcal{F}}\] 
does not depend on the chosen lifting (see \cite{g}).
We denote by $\delta\nu$ this class and call it the infinitesimal invariant.
It is usefull to define also the dual version of $\delta\nu$. 
Let us consider the transpose of the Gauss-Manin,
$\nabla^{t}:\mathcal{F}\otimes T_{U}\rightarrow \mathcal{H}^{0,1}$. 
Choose a reference point \mbox{$0\in U$} and denote $C=C_{0}$. 
Then $\nabla^{t}:H^{0}(C,\omega_{C})\otimes H^{1}(C,T_{C})\rightarrow H^{1}(C,\mathcal{O}_{C})$ is given by
$\nabla^{t}(\sum_{i}\omega_{i}\otimes\xi_{i})=\sum_{i}\nabla_{\xi_{i}}\omega_{i}=\sum_{i}\xi\cdot\omega_{i}$.
By the duality $ \frac{\displaystyle{\mathcal{H}^{0,1}\otimes\Omega_{U}^{1}}}{\displaystyle{\nabla \mathcal{F}}} \cong (\ker(\nabla^{t}))^{*}$, 
we can consider $\delta\nu$ as an an element of this last vector space. Over the point $0$ we have:
\[\delta\nu(0)(\sum_{i}\omega_{i}\otimes \xi_{i})=\sum_{i}\int_{C}\nabla_{\xi_{i}}\widetilde{\nu}\wedge\omega_{i},\]
 where  $\sum_{i}\xi_{i}\cdot\omega_{i}=0$. This is Griffiths' definition of the infinitesimal invariant \cite{grif}.
The first tool in our computation will be the following
\begin{teo}\textup{\cite[pp. 292-293]{grif}}\label{calc} Let $\xi \otimes \omega\in H^{1}(C,T_{C})\otimes H^{0}(C,\omega_{C})$ be such that
 $\xi\cdot\omega=0$ in $H^{1}(C,\mathcal{O}_{C})$. Choose $h\in C^{\infty}(C)$ such that $\xi\cdot\omega=\overline{\partial}h$ 
and call $D_{0}=\sum_{i=1}^{l}p_{i}-q_{i}$.
\\
Then the number $\sum_{i=1}^{l}h(p_{i})-h(q_{i})$ depends only on $\xi\otimes\omega$ and we have 
\[
\delta\nu(0)(\xi\otimes\omega)=\sum_{i=1}^{l}h(p_{i})-h(q_{i}).
\]
\end{teo}

The second tool we need is a computation by Voisin \cite{voi} which we now describe in our particular case.
\\
Let $\pi:\mathcal{C}\rightarrow \Delta$ be a smooth family of curves over the unit disc in $\mathbb{C}$ 
and let $\mathcal{D}$ be as above. 
There is a short exact sequence over $C$:
\[ 0\rightarrow\mathcal{O}_{C}\rightarrow
\Omega^{1}_{\mathcal{C}|C}\otimes\pi^{*}T_{\Delta,0}\rightarrow 
\omega_{C}\otimes \pi^{*}T_{\Delta,0}\rightarrow 0,\]
which gives in cohomology
\[ H^{0}(C,\Omega^{1}_{\mathcal{C}|C}\otimes\pi^{*}T_{\Delta,0})\stackrel{\gamma}{\rightarrow}H^{0}(C,\omega_{C})
\otimes T_{\Delta,0}\stackrel{\delta}{\rightarrow} H^{1}(C,\mathcal{O}_{C}).\]
If $\omega\otimes\eta\in H^{0}(C,\omega_{C})\otimes T_{\Delta,0}$ 
is such that $\delta(\omega\otimes\eta)=0$, then there exists 
$\Omega=\hat{\omega}\otimes\eta\in H^{0}(C,\Omega^{1}_{\mathcal{C}|C}\otimes\pi^{*}T_{\Delta,0})$ 
such that $\gamma(\Omega)=\omega\otimes\eta$.
\\
Write $D_{0}=D_{1,0}-D_{2,0}=\sum_{i=1}^{l}p_{i}-\sum_{i=1}^{l}q_{i}$, and 
consider, for $i=1,2$, the image of $\Omega$ under the map $F_{i}:H^{0}(C,\Omega^{1}_{\mathcal{C}|C}\otimes\pi^{*}T_{\Delta,0})
\rightarrow H^{0}(D_{i,0},\mathcal{O}_{D_{i,0}})$ obtained as the composition of the following maps
\[ H^{0}(C,\Omega^{1}_{\mathcal{C}|C}\otimes\pi^{*}T_{\Delta,0})\rightarrow H^{0}(C,\Omega^{1}_{D_{i}|D_{i,0}}
\otimes\pi^{*}T_{\Delta,0})\rightarrow H^{0}(D_{i,0},\mathcal{O}_{D_{i,0}})\simeq\mathbb{C}^{l};\]
i.e. $F_{1}(\Omega)=(\hat{\omega}(p_{1}),\cdots,\hat{\omega}(p_{l}))$, $F_{2}(\Omega)=(\hat{\omega}(q_{1}),\cdots,\hat{\omega}(q_{l}))$. Then we have: 
\begin{equation}
\label{voi}
\delta\nu(0)(\omega\otimes\eta)=\sum_{i=1}^{l}\hat{\omega}(p_{i})-\hat{\omega}(q_{i})
\end{equation} (see \cite[5.5 and 5.8]{voi} where it is proved in much more generality).

\section{The Torelli theorem}\label{ex}
Let $C$ be a smooth generic curve of genus $g\geqslant 4$. Recall that
\[ W_{g-1}^{i}(C)=\{L\in\Pic^{g-1}(C)\textit{ }|\textit{ }h^{0}(L)\geqslant i+1\},
\]
and that its smooth locus is
\[
W_{g-1,\textrm{sm}}^{i}(C)=\{L\in\Pic^{g-1}(C)\textit{ }|\textit{ }h^{0}(L)= i+1\}.\]
Let us consider a line bundle $L\in W_{g-1,\textrm{sm}}^{1}(C)$ without base point
 and its residue line bundle $\omega_{C}\otimes L^{-1}\in W_{g-1,\textrm{sm}}^{1}(C)$.
Notice that since $C$ is generic $L^{\otimes 2}\neq \omega_{C}$.
\\
We fix two basis $\{s_{0},s_{1}\}$, $\{t_{0},t_{1}\}$ of $H^{0}(C,L)$ and $H^{0}(C,\omega_{C}\otimes L^{-1})$ respectively. 
\\
Denoting by  
\[\mu: H^{0}(C,L)\otimes H^{0}(C,\omega_{C}\otimes L^{-1})\rightarrow H^{0}(C,\omega_{C})\] 
the Petri map, we consider the holomorphic forms $\omega_{ij}=\mu(s_{i}\otimes t_{j})$. 
The vector space $V=\spa\{\omega_{00},\omega_{01},\omega_{10},\omega_{11}\}$ is a four dimensional 
vector subspace of $H^{0}(C,\omega_{C})$ without base point. 
We then have a holomorphic map $\phi: C\rightarrow\mathbb{P}V^{*}\cong\mathbb{P}^{3}$ and $\phi(C)$ lies on the quadric 
$Q=\{\omega_{00}\omega_{11}-\omega_{01}\omega_{10}=0\}$. 
Furthermore $C$ is birational to $\phi(C)$ \cite[prop. 8.33 p. 834]{ar2}.
\\
Let $\pi:\mathcal{C}\rightarrow B$ be the Kuranishi family of $C$. Restricting $B$ if necessary, we can suppose that there exist sections 
$p_{i}:B\rightarrow \mathcal{C}$, 
$i=1,\cdots,g-1$, such that $D_{t}=\sum_{i=1}^{g-1}p_{i}(t)$ verifies $L\cong \mathcal{O}_{C}(D_{0})$ and $H^{0}(C_{t},\mathcal{O}_{C_{t}}(D_{t}))=2$ 
for every 
$t\in B$, i.e. $L_{t}=\mathcal{O}_{C_{t}}(D_{t})\in W_{g-1,\textrm{sm}}^{1}(C_{t})$. We can regard $\{(C_{t},L_{t})\}_{t\in B}$ as a deformation of $(C,L)$. 
Since $2D_{t}-K_{C_{t}}$ has degree zero, we then get a normal function $\nu(t)=AJ_{C_{t}}(2D_{t}-K_{C_{t}})$ ($K_{C_{t}}$ denotes a canonical divisor of $C_{t}$).
\\
Our aim is to prove that the infinitesimal $\delta\nu(0)$ reconstructs the curve \mbox{$\phi(C)\subset\mathbb{P}^{3}$.}
 \\
 As first step, we construct, for every point $q\in Q^{\prime}:=Q\setminus\textrm{Sing}(\phi(C))$, 
 an holomorphic form $\omega_{q}\in V$ and an element $\xi_{q}\in H^{1}(C,T_{C})$ 
 such that $\xi_{q}\cdot\omega_{q}=0$.
\begin{oss} The normal function vanishes in $t$ if $L_{t}$ is a theta characteristic, i.e. $L_{t}^{\otimes 2}\cong \omega_{C_{t}}$. 
\end{oss}
\begin{oss} We could have defined the normal function over the relative Brill-Noether variety $\mathcal{W}^{1}_{g-1}$ (see \cite[chap. XXI]{ar2}). 
Its tangent space in fact parametrizes first order deformations of $(C,L)$ that preserves sections of $L$. We decided to work instead over the Kuranishi family $B$, 
to stress the fact that our calculation of $\delta\nu(0)$ depends only on first order deformations of $C$, which are parametrized by $T_{B,0}\cong H^{1}(C,T_{C})$.
\end{oss} 

\subsection{The form $\omega_{q}$}
\label{form}
Consider the Segre map $S: \mathbb{P}H^{0}(L)\times \mathbb{P}H^{0}(\omega_{C}\otimes L^{-1})\rightarrow \mathbb{P}V\cong\mathbb{P}^{3}$. 
The basis chosen above give projective coordinates $x_{0},x_{1}$ on $\mathbb{P}H^{0}(L)\simeq\mathbb{P}^{1}$ and 
$y_{0},y_{1}$ on \mbox{$\mathbb{P}H^{0}(\omega_{C}\otimes L^{-1})\simeq\mathbb{P}^{1}$}.
Let us also fix coordinates $z_{ij}$ on $\mathbb{P}^{3}$.
Then $Q=\{z_{00}z_{11}-z_{01}z_{10}=0\}$ is the image of $S$.
\\
Let us consider a point $q\in Q$, say \mbox{$q=S((a_{0}:a_{1}),(b_{0}:b_{1}))$}.
\\
Call $H$ the hyperplane of $\mathbb{P}^{3}$ containing the two rulings $l_{1},l_{2}$ passing through the point $q$. 
\\
We have the equations $l_{1}=\{a_{1}z_{00}-a_{0}z_{10}=a_{1}z_{01}-a_{0}z_{11}=0\}$,
$l_{2}=\{b_{1}z_{00}-b_{0}z_{01}=b_{1}z_{10}-b_{0}z_{11}=0\}$ and $H=\{F=0\}$, 
where $F\in H^{0}(\mathcal{O}_{\mathbb{P}^{3}}(1))$ is the polynomial \mbox{$a_{1}b_{1}z_{00}-a_{0}b_{1}z_{10}-a_{1}b_{0}z_{01}+a_{0}b_{0}z_{11}$.}
Notice that we have 
\[S^{*}(F)=(a_{1}x_{0}-a_{0}x_{1})(b_{1}y_{0}-b_{0}y_{1}).\]
Let us define $s=a_{1}s_{0}-a_{0}s_{1}$, $t=b_{1}t_{0}-b_{0}t_{1}$ and 
\[\omega_{q}=\mu(s\otimes t)=a_{1}b_{1}\omega_{00}-a_{0}b_{1}\omega_{10}-a_{1}b_{0}\omega_{01}+a_{0}b_{0}\omega_{11}.\]

\begin{oss}\label{0}
  1) If $q=\phi(p)$ we have that $p\in Z(s)\cap Z(t)$, where $Z(s)$ and $Z(t)$ are the zero divisors of $s$ and $t$.
\\
In fact $p$ is not a base point of $|\omega_{C}\otimes L^{-1}|$ so $t_{0}(p)$ and $t_{1}(p)$ are not both zero.
Then $s(p)=0$ if and only if $s(p)t_{0}(p)=s(p)t_{1}(p)=0$, i.e. $a_{1}\omega_{00}(p)-a_{0}\omega_{10}(p)=a_{1}\omega_{01}(p)-a_{0}\omega_{11}(p)=0$.
But this last equation is verified because $q\in l_{1}$.
\\
In the same way one can check that $t(p)=0$.
\\
2) An argument similar to the above also yelds that $\phi^{*}(l_{1|\phi(C)})=Z(s)$.
\end{oss}

\subsection{The deformation $\xi_{q}$}
\label{symm}
Let $q$ be a point in $Q^{\prime}=Q\setminus\textrm{Sing}(\phi(C))$ and consider the sections $s,t$ constructed before.
Let us also define $\phi(C)_{sm}:=\phi(C)\setminus\textrm{Sing}(\phi(C))$, the smooth locus of $\phi(C)$.
Choose sections $s^{\prime}\in H^{0}(C,L)$ and \mbox{$t^{\prime}\in H^{0}(C,\omega_{C}\otimes L^{-1})$} such that $\spa\{s,s^{\prime}\}=H^{0}(C,L)$ 
and $\spa\{t,t^{\prime}\}=H^{0}(C,\omega_{C}\otimes L^{-1})$. 
For notational convenience we denote $\mu(s\otimes t)=st$ and similarly for the other cup-products.
We define $W=\spa\{st,st^{\prime},s^{\prime}t\}\subset V$ and $\Ann W=\{\xi\in H^{1}(T_{C}) \ | \ \xi\cdot W=0\}$.
 \begin{lem}
We have $\dim(\Ann W)=1$.
\end{lem}
\begin{proof}
Consider the spaces 
\[ 
U=\I(m:W\otimes H^{0}(C,\omega_{C})\rightarrow H^{0}(C,\omega_{C}^{\otimes2})) \textrm{   and} 
\]
\[
\Ann U=\{\xi\in H^{1}(C,T_{C})\textit{ }|\textit{ }_{H^{1}(T_{C})}\langle\xi,U\rangle_{H^{0}(\omega_{C}^{\otimes 2})}=0\}.
\]
From the equality
 \[_ {H^{1}(T_{C})}\langle\xi,\omega\omega^{\prime}\rangle_{H^{0}(\omega_{C}^{\otimes 2})}=
\textrm{ } _{H^{1}(\mathcal{O}_{C})}\langle\xi\cdot\omega,\omega^{\prime}\rangle_{H^{0}(\omega_{C})}\]
it follows that $\Ann W=\Ann U$.
In order to compute $\dim(\Ann U)$ we set 
\[
W_{1}=\spa\{st,s^{\prime}t\}, \textrm{  } W_{2}=\spa\{st,st^{\prime}\} \textrm{  and} 
\]
\[
U_{i}=\I(m:W_{i}\otimes H^{0}(C,\omega_{C})\rightarrow H^{0}(C,\omega_{C}^{\otimes 2})) \textrm{  for $i=1,2$.}
\]
Recall that $s$ and $s^{\prime}$ have no common zeroes.
Using the base-point-free-pencil trick \cite[p.~126]{ar} we see that 
$\dim(U_{1})=\dim(H^{0}(C,L)\otimes H^{0}(C,\omega_{C}))-\dim(\ke m)=2g-h^{0}(\omega_{C}\otimes L^{-1})=2g-2$ and, in the same way, 
we find $\dim(U_{2})=2g-2$.
\\
Now we calculate $\dim(U_{1}\cap U_{2})$. 
Let us suppose first that $q\in\phi(C)_{sm}$.
Then $\phi^{-1}(q)$ is a single point $p$ and, by Remark \ref{0}, $s$ and $t$ both vanishes in $p$.
Write $Z(s)=A+p$ and $Z(t)=B+p$ and notice that 
\[
U_{1}\subset H^{0}(C,\omega_{C}^{\otimes2}(-p-B))
\] 
\[
U_{2}\subset H^{0}(C,\omega_{C}^{\otimes2}(-p-A))
.\]
Since $h^{0}(\omega_{C}^{\otimes2}(-p-B)))=2g-2=\dim U_{1}$ we get $U_{1}= H^{0}(C,\omega_{C}^{\otimes2}(-p-B))$, and
simililarly $U_{2}=H^{0}(C,\omega_{C}^{\otimes2}(-p-A))$.
Then we have $\dim(U_{1}\cap U_{2})=h^{0}(\omega_{C}^{\otimes2}-p-A-B)=h^{0}(\omega_{C}(p))=g$ (we are using that $\mathcal{O}_{C}(2p+A+B)=\omega_{C}$).
Therefore $\dim U=\dim U_{1}+\dim U_{2}-\dim(U_{1}\cap U_{2})=3g-4$ and by duality $\dim(\Ann U)=1$.

Let us suppose now that $q\in Q\setminus\phi(C)$.
Then $s$ and $t$ have no common zeroes.
Reasoning as above we see that $U_{1}=H^{0}(C,\omega_{C}^{\otimes2}(-E))$ and $U_{2}=H^{0}(C,\omega_{C}^{\otimes2}(-D))$ where $D=Z(s)$ and $E=Z(t)$.
Then $\dim(U_{1}\cap U_{2})=h^{0}(\omega_{C}^{\otimes2}-D-E)=h^{0}(\omega_{C})=g$ and then $\dim U=3g-4$.
 \end{proof}
\begin{oss}
Although the space $W$ depends on the choice of the sections $s'$ and $t'$, the space $\Ann W$ does not.
\end{oss}
 
Given a point $q\in Q^{\prime}$ we call $\xi_{q}$ the element of $H^{1}(C,T_{C})$ that generates $\Ann W$.
We get in this way a rational map $\varphi:Q\dashrightarrow\mathbb{P}H^{1}(C,T_{C})$ which is defined on $Q^{\prime}$. 
We want an explicit description of the first order \mbox{deformation $\xi_{q}$.}

If $q=\phi(p)$ is a smooth point of $\phi(C)$ then $\xi_{q}=\theta_{p}$, the Schiffer variation in $p$.
In fact, $\ker(\theta_{p}:H^{0}(C,\omega_{C})\to H^{1}(C,\mathcal{O}_{C}))=H^{0}(C,\omega_{C}(-p))$ (see \cite[p. 275]{grif}) and 
$W\subset H^{0}(C,\omega_{C}(-p))$ because $s(p)=t(p)=0$. 
Then $\theta_{p}\cdot W=0$.

Let us consider a point $q\in Q\setminus\phi(C)$.
If we denote $F=L\oplus(\omega_{C}\otimes L^{-1})$ we have a commutative diagram
\begin{equation}
\label{diag}
\xymatrix{0\ar[r] & T_{C}\ar[r]^{(t,-s)}\ar[d]^{\omega_{q}} & F^{*}\ar[r]^{\binom{s}{t}}\ar[d]^{\omega_{q}} & \mathcal{O}_{C}\ar[r]\ar[d]^{\omega_{q}} & 0 \\
0\ar[r] & \mathcal{O}_{C}\ar[r]^{(s,t)} & F\ar[r]^{\binom{t}{-s}} & \omega_{C}\ar[r] & 0.
}
\end{equation} 
The first row defines an element of Ext$^{1}(\mathcal{O}_{C},T_{C})$. We call $\delta:H^{0}(C,\mathcal{O}_{C})\rightarrow H^{1}(C,T_{C})$ and 
$\delta^{\prime}:H^{0}(C,\omega_{C})\rightarrow H^{1}(C,\mathcal{O}_{C})$ the coboundaries of the first and second row. 
Let us define $\xi_{q}=\delta(1)$ and recall the map $\delta^{\prime}$ is given by cupping with $\xi_{q}$.
Notice that the forms $\omega_{q}=st$, $st^{\prime}$ and $s^{\prime}t$ have lifting to \mbox{$H^{0}(C,F)\simeq H^{0}(C,L)\oplus H^{0}(C,\omega_{C}\otimes L^{-1})$}
given by, respectively, $(s,0)$, $(s^{\prime},0)$, $(0,-t^{\prime})$. 
Then by exacteness of the cohomology sequence $0=\delta^{\prime}(W)=\xi_{q}\cdot W$.
\subsection{Computation of $\delta\nu$}
Let us call for brevity $\mathbb{P}=\mathbb{P}H^{1}(C,T_{C})$,  and consider 
\[
\Sigma=\{((\xi),(\omega))\in\mathbb{P}\times\mathbb{P}V\textrm{ }|\textrm{ }\xi\cdot\omega=0\},\]
where we call $(\xi)$ the class of $\xi$ in $\mathbb{P}$ and similarly for $(\omega)$.
Inside $\Sigma$ we have the variety $X=\{((\xi_{q}),(\omega_{q})) \ | \ q\in Q^{\prime}\}$. 
Let us call $\pi_{1}$ and $\pi_{2}$ the projections from $\mathbb{P}\times\mathbb{P}V$ onto 
$\mathbb{P}$ and $\mathbb{P}V\cong\mathbb{P}^{3}$ respectively.
Then $\delta\nu(0)$ can be thought as a map 
\begin{equation}
\label{sec}
 \delta\nu(0):(\pi^{*}_{1}\mathcal{O}_{\mathbb{P}}(-1)\otimes\pi_{2}^{*}\mathcal{O}_{\mathbb{P}^{3}}(-1))_{|X}
\rightarrow\mathcal{O}_{X}.
\end{equation}
\begin{teo}\label{i.i.}
We have
\begin{align}
\delta\nu(0)(\xi_{q}\otimes\omega_{q}) & =0 \textrm{ if }q\in\phi(C)_{sm} \\
\delta\nu(0)(\xi_{q}\otimes\omega_{q}) & \neq 0  \textrm{ if }q\in Q\setminus\phi(C) .
\end{align}
\end{teo}

\begin{proof} 
\textit{Case 1. $q\in\phi(C)_{sm}$ }
\\
Recall that for $(U,z)$ a coordinate chart centered in $p$, locally in $U$ we have that $\theta_{p}=\frac{\overline{\partial}\rho(z)}{z}\frac{\partial}{\partial z}$, 
where $\rho$ is a bump function in $p$. In this coordinate chart we can also write $\omega_{q}=zf(z)dz$ for some holomorphic function $f$. 
Then $\xi_{q}\cdot\omega_{q}=\overline{\partial}(f\rho)$.
\\
Set $Z(s)=\sum_{i=1}^{g-1}p_{i}$ and $Z(t)=\sum_{i=1}^{g-1}q_{i}$. Since $p\in Z(s)\cap Z(t)$ we can suppose that $p=p_{1}=q_{1}$.
Then by Theorem \ref{calc} we have \[
\delta\nu(0)(\xi_{q}\otimes\omega_{q})=\sum_{i=1}^{g-1}f(p_{i})\rho(p_{i})-f(q_{i})\rho(q_{i})=0,
\]
 where the second equality follows by the fact that $\rho$ is a bump function.
\medskip
\\
\textit{Case 2. $q\in Q\setminus\phi(C)$}
\\
Let $\Delta\subset B$ be the unit disc in $\mathbb{C}$ such that
 $T_{\Delta,0}=\spa\{\xi_{q}\}$ (here we are using that $T_{B,0}\simeq H^{1}(C,T_{C})$). 
 By pull-back we can suppose that our family of curves is $\pi:\mathcal{C}\rightarrow\Delta$.
Recall from Section \ref{1} that we have the exact sequence 
\[ H^{0}(C,\Omega^{1}_{\mathcal{C}|C}\otimes\pi^{*}T_{\Delta,0})\simeq H^{0}(C,\Omega^{1}_{\mathcal{C}|C})
\otimes T_{\Delta,0}\stackrel{\gamma}{\rightarrow}H^{0}(C,\omega_{C})\otimes T_{\Delta,0}\stackrel{\delta}{\rightarrow} H^{1}(C,\mathcal{O}_{C}).\]
Remember that we have constructed $s\in H^{0}(C,L)$ and $t\in H^{0}(C,\omega_{C}\otimes L^{-1})$ and that \mbox{$\omega_{q}=\mu(s\otimes t)$},
$Z(s)=\sum_{i=1}^{g-1}p_{i}$, $Z(t)=\sum_{i=1}^{g-1}q_{i}$.
\\
We write, locally near $p_{i}$, $\omega_{q}=f_{i}(z_{i})dz_{i}$, and 
near $q_{i}$, $\omega_{q}=\tilde{f}_{i}(w_{i})dw_{i}$. 
Now since $\delta(\omega_{q}\otimes\xi_{q})=\xi_{q}\cdot\omega_{q}=0$, $\omega_{q}\otimes\xi_{q}$ lifts to 
$\Omega=\hat{\omega}_{q}\otimes\xi_{q}$, with $\hat{\omega}_{q}\in H^{0}(C,\Omega^{1}_{\mathcal{C}|C})$.
\\
Then we have $\hat{\omega}_{q}=f_{i}(z_{i})dz_{i}+g_{i}(z_{i})dt_{i}$ near $p_{i}$ and 
$\hat{\omega}_{q}=\tilde{f}_{i}(w_{i})dw_{i}+\tilde{g}_{i}(w_{i})dt_{i}$ near $q_{i}$. 
\\
With the notations of Section \ref{1} we have $D_{1,0}=\sum_{i=1}^{g-1}p_{i}$ and $D_{2,0}=\sum_{i=1}^{g-1}q_{i}$ and 
\[
F_{1}(\hat{\omega}_{q}\otimes\xi_{q})=(g_{1}(p_{1}),\cdots,g_{g-1}(p_{g-1}))\textrm{  and}
\]
\[
F_{2}(\hat{\omega}_{q}\otimes\xi_{q})=(\tilde{g}_{1}(q_{1}),\cdots,\tilde{g}_{g-1}(q_{g-1})),
\]
because $f_{i}(p_{i})=\tilde{f}_{i}(q_{i})=0$ for $i=1,\cdots g-1$. So by (\ref{voi}) we have 
\begin{equation}
\label{formula}
\delta\nu(0)(\omega_{q}\otimes\xi_{q})=\sum_{i=1}^{g-1}g_{i}(p_{i})-\tilde{g}_{i}(q_{i}). 
\end{equation}

\begin{lem}\label{3.7}
There exist $c\in\mathbb{C}$ such that $g_{i}(p_{i})=c$ and $\tilde{g}_{i}(q_{i})=c+1$ for $i=1,\cdots,g-1$.
\end{lem}\label{lemma}
\begin{proof}
We have the commutative diagram
\begin{equation}
\label{big}
\xymatrix{
& 0 \ar[d] & 0\ar[d] & 0 \ar[d] & \\
0 \ar[r] & T_{C} \ar[r]\ar[d]^{\omega_{q}} & T_{\mathcal{C}|C}\ar[r]\ar[d]^{\omega_{q}} & \mathcal{O}_{C} \ar[r]\ar[d]^{\omega_{q}} & 0 \\
0 \ar[r] & \mathcal{O}_{C} \ar[r]\ar[d]^{r} & \Omega^{1}_{\mathcal{C}|C}\ar[r] \ar[d]^{r} & \omega_{C} \ar[r] \ar[d]^{r} & 0 \\
0 \ar[r] & \mathcal{O}_{Z(\omega_{q})} \ar[r]^{i}\ar[d] & \Omega^{1}_{\mathcal{C}|Z(\omega_{q})} \ar[r]\ar[d] & \omega_{C|Z(\omega_{q})} \ar[r]\ar[d] & 0 \\ 
& 0  & 0 & 0  &
,}
\end{equation}
where $\mathcal{O}_{Z(\omega_{q})}\simeq\bigoplus_{i=1}^{g-1}\mathbb{C}_{p_{i}}\oplus \bigoplus_{i=1}^{g-1}\mathbb{C}_{q_{i}}$.
It gives the following diagram in cohomology
\[
\xymatrix{
\mathbb{C} \ar[r]\ar[d]^{r} & H^{0}(C,\Omega^{1}_{\mathcal{C}|C}) \ar[r] \ar[d]^{r} & H^{0}(C,\omega_{C})  \\
\mathbb{C}^{2g-2} \ar[r]^{\hspace{-1cm}i} & H^{0}(Z(\omega_{q}),\Omega^{1}_{\mathcal{C}|Z(\omega_{q})})  & &
.}
\]
It is clear that $r(\hat{\omega}_{q})=i((g_{1}(p_{1}),\cdots,g_{g-1}(p_{g-1}),\tilde{g}_{1}(q_{1}),\cdots,\tilde{g}_{g-1}(q_{g-1})))$. 
Note that the association $\omega_{q}\mapsto (g_{1}(p_{1}),\cdots,g_{g-1}(p_{g-1}),\tilde{g}_{1}(q_{1}),\cdots,\tilde{g}_{g-1}(q_{g-1}))\in\mathbb{C}^{2g-2}$
 depends on the choice of a lifting $\hat{\omega}_{q}$, so it is well defined modulo $r(\mathbb{C})$.
\\
Under the isomorphism $F^{*}\simeq T_{\mathcal{C}|C}$ diagram (\ref{big}) becomes
\begin{equation}
\label{big'}
\xymatrix{
& 0 \ar[d] & 0\ar[d] & 0 \ar[d] & \\
0 \ar[r] & T_{C} \ar[r]\ar[d]^{\omega_{q}} & F^{*} \ar[r]\ar[d]^{\omega_{q}} & \mathcal{O}_{C} \ar[r]\ar[d]^{\omega_{q}} & 0 \\
0 \ar[r] & \mathcal{O}_{C} \ar[r]^{(s,t)}\ar[d]^{r} & F \ar[r]^{\binom{t}{-s}} \ar[d]^{r} & \omega_{C} \ar[r] \ar[d]^{r} & 0 \\
0 \ar[r] & \mathcal{O}_{Z(\omega_{q})} \ar[r]^{i}\ar[d] & F_{|Z(\omega_{q})} \ar[r]\ar[d] & \omega_{C|Z(\omega_{q})} \ar[r]\ar[d] & 0 \\ 
& 0  & 0 & 0  &
}
\end{equation}

In cohomology $\omega_{q}=st$ lifts to $(s,0)\in H^{0}(C,F)$. It restricts to $(s_{|Z(\omega_{q})},0)\in H^{0}(Z(\omega_{q}),F_{|Z(\omega_{q})})$, 
which lifts to $(0,\cdots,0,1,\cdots,1)\in\mathbb{C}^{g-1}\oplus\mathbb{C}^{g-1}$. 
\\ So in this case we have that $\omega_{q}\mapsto (0,\cdots,0,1,\cdots,1)\in\mathbb{C}^{2g-2}$.
\\
Compairing with what we found above we conclude \[
(g_{1}(p_{1}),\cdots,g_{g-1}(p_{g-1}),\tilde{g}_{1}(q_{1}),\cdots,\tilde{g}_{g-1}(q_{g-1}))=(0,\cdots,0,1,\cdots,1)\textrm{ mod }r(\mathbb{C}),\] 
which is the thesis.
\end{proof}
Lemma \ref{3.7} together with (\ref{formula}) completes the proof of the Proposition.
\end{proof} 
\begin{cor}\label{lb} The infinitesimal invariant reconstructs the couple $(C,L)$.
\end{cor}
\begin{proof}
Recall that, by construction, the hyperplane $\{\omega_{q}=0\}\subset \mathbb{P}V^{*}$ is the projective tangent space $T_{Q,q}$ of 
$Q=\{\omega_{00}\omega_{11}-\omega_{01}\omega_{10}=0\}$ at the point $q$
(see Section \ref{form}).
\\
If we consider $\mathbb{P}V$ as the space of hyperplanes of $\mathbb{P}V^{*}$,
then $\{\omega_{q} \ | \ q\in Q\}\cong\{T_{Q,q} \ | \ q\in Q \}=Q^{*}\subset\mathbb{P}V$, the dual of the quadric $Q$.
It is well known that $Q^{*}\cong Q$ and then $\{\omega_{q} \ | \ q\in Q\}\cong Q\subset\mathbb{P}^{3}$.
Let us define $Z:=\mbox{$\{((\xi_{q}),(\omega_{q}))\in X \ | \ \delta\nu(0)(\xi_{q}\otimes\omega_{q})=0)\}$}$ the zero locus of the map (\ref{sec}). 
It follows from Proposition \ref{i.i.} that $\pi_{2}(Z)=\phi(C)_{sm}\subset\mathbb{P}^{3}$ and then taking the closure we get $\overline{\pi_{2}(Z)}=\phi(C)$.
Therefore $\delta\nu(0)$ reconstructs $\phi(C)$ and, since $\phi$ is birational onto $\phi(C)$, it also \mbox{reconstructs $C$.}

To recover $L$, it suffices to recall that by remark \ref{0} we have $\phi^{*}(l_{1|\phi(C)})=Z(s)$ (notations as in section \ref{form}), 
and that obviously $\mathcal{O}_{C}(Z(s))\cong L$.
\end{proof}
\begin{oss}
When $g=4$ we have that  $\phi(C)$ is the canonical image of $C$. We thus obtain Griffiths' result \cite[pp. 298-302]{grif}.
\end{oss}
\textbf{Acknowledgements.} I would like to thank Prof. G.P. Pirola for having suggested me the problem and for his support in this work. 
\\ 
It is a pleasure to thank Prof. A. Collino. He pointed out an inaccuracy in the first version of the paper and also suggested how to correct it.

\noindent Emanuele\ Raviolo\\
Dipartimento di Matematica, Universit\`a di Pavia \\
via Ferrata 1 \\
 27100 Pavia - Italy \\
 emanuele.raviolo@unipv.it

\end{document}